\documentclass{amsart}
\usepackage{amsfonts,tabularx,graphicx}
\usepackage{blkarray}
\usepackage{float}

\newtheorem{theorem}{Theorem}[section]

\newtheorem{corollary}[theorem]{Corollary}
\theoremstyle{definition}
\newtheorem{definition}[theorem]{Definition}
\newtheorem{example}[theorem]{Example}

\theoremstyle{remark}
\newtheorem{remark}[theorem]{Remark}
\numberwithin{equation}{section}
\newcommand{\F}{\mathbb{F}}

\begin{document}
\title[Self-dual codes via a Baumert-Hall array]{New extremal binary
self-dual codes from a Baumert-Hall array}
\author{Abidin Kaya}
\address{Sampoerna Academy, L'Avenue Campus, 12780, Jakarta, Indonesia}
\email{abidin.kaya@sampoernaacademy.sch.id}
\author{Bahattin Yildiz}
\address{Department of Mathematics and Statistics, Northern Arizona
University, Flagstaff, AZ 86001, USA}
\email{bahattin.yildiz@nau.edu} \subjclass[2010]{Primary 94B05,
94B99; Secondary 11T71, 13M99} \keywords{extremal self-dual codes,
codes over rings, Gray maps, Baumert-Hall array, Baumert-Hall array,
extension theorems}

\begin{abstract}
In this work, we introduce new construction methods for self-dual codes using a Baumert-Hall array. We apply the constructions over the alphabets $\mathbb{F}_{2}$ and $%
\mathbb{F}_{4}+u\mathbb{F}_{4}$ and combine them with extension
theorems and neighboring constructions. As a result, we construct 46
new extremal binary self-dual codes of length 68, 26 new best known
Type II codes of length 72 and 8 new extremal Type II codes of
length 80 that lead to new $3-(80,16,665)$ designs. Among the new
codes of length 68 are the examples of codes with the rare
$\gamma=5$ parameter in $W_{68,2}$. All these new codes are
tabulated in the paper.
\end{abstract}

\maketitle

\section{Introduction}

Finding extremal binary self-dual codes with new weight enumerators
has been a topic of considerable interest in the Coding Theory
community for decades now. There are motivating factors for this
interest that come, in part, from the connection of self-dual codes
to structures such as designs, lattices and invariant polynomials.
The Assmus-Mattson theorem, for example, establishes a strong
connection between self-dual codes and designs. This connection was
used in \cite{kayayildiz} to find new 3-designs from Type II
extremal binary self-dual codes of length 80.

To understand some of the construction methods for self-dual codes,
we recall that a binary self-dual code of length $2n$ is generated
by (up to equivalence) a matrix of the form $[I_n|A]$, where $A$ is
an $n\times n$ block matrix. In the absence of any other algorithm,
the randomness of the matrix $A$ makes for an impractical search
field of $2^{n^2}$. Even when the orthogonality relations are
factored in, we still have a search field of size $2^{n(n+1)/2} $,
which is still considerably far from being practical and impractical
with the current technology. This is partly the reason why the
existence/non-existence of the Type II extremal binary self-dual
code of length 72 is still an open problem for coding theorists.

The difficulties in a general search for self-dual codes have led
researchers to utilize particular types of matrices in an effort to
reduce the search field. Most common techniques in the literature
use some variation of a construction that uses circulant matrices.
The double-circulant, bordered double circulant, four circulant
constructions are all special constructions that use circulant
matrices. In most instances
of these constructions, the search field is reduced from $2^{O(n^2)}$ to $%
2^{O(n)}$, which is a big improvement in the implementation of
search algorithms.

Another approach that has been used in the literature is combining
the constructions mentioned above over rings that are equipped with
orthogonality-preserving Gray maps. The algebraic structure of rings
and the nature of the Gray map lead to binary self-dual codes with a
particular automorphism group that may have been missed by the
previous constructions. This has been successfully applied in works
such as \cite{gildea}, \cite{kaya}, \cite{shortkharaghani},
\cite{kayayildiz}, \cite{pasa}. In\cite{fsd}, these ideas were
applied for constructing formally self-dual codes of high minimum
distances.

In this work, we describe a construction coming from a Baumert-Hall
array to find binary self-dual codes. Similar matrices coming from
Discrete Mathematics have been used to construct self-dual codes
before, e.g. \cite{shortkharaghani}, \cite{tonchev}.  We apply the
constructions over the
binary field as well as the rings $\mathbb{F}_2+u\mathbb{F}_2$ and $\mathbb{F%
}_4+u\mathbb{F}_4$, which are equipped with orthogonality-preserving
Gray maps. The constructions turn out to be efficient as we are able
to find many new extremal binary self-dual codes. In particular we
find 46 extremal binary self-dual codes of length 68 with new weight
enumerators, including the examples of codes with the rare
$\gamma=5$ parameter in $W_{68,2}$. The existence of codes with
$\gamma=8$ and $\gamma=9$ is still an open problem. We also find 26
new best known Type II codes of length 72 and 8 new extremal Type II
codes of length 80 that lead to new $3-(80,16,665)$ designs.

The rest of the work is organized as follows. In section 2, we give
the preliminaries on circulant matrices, self-dual codes, alphabets
we use and Baumert-Hall arrays. In section 3, we describe the
construction method for self-dual codes. In section 4, we give the
numerical results and tables corresponding to the codes constructed.
We finish the paper with concluding remarks and directions for
possible future research.

\section{Preliminaries}

\subsection{Matrices}

The circulant matrices are a special type of matrices that are used
heavily in many constructions for extremal self-dual codes. We
recall that a circulant matrix is a square matrix where each row is
a right-circular shift of the previous row. In other words, if
$\overline{r}$ is the first row, a typical circulant matrix is of
the form
\begin{equation}
\left[
\begin{array}{c}
\overline{r} \\ \hline \sigma(\overline{r}) \\ \hline
\sigma^2(\overline{r}) \\ \hline \vdots \\ \hline
\sigma^{n-1}(\overline{r})%
\end{array}
\right],
\end{equation}
where $\sigma$ denotes the right circular shift. It is clear that,
with $T$ denoting the permutation matrix corresponding to the
$n$-cycle $(123...n)$, a circulant matrix with first row $(a_1,a_2,
\dots, a_n)$ can be expressed as a polynomial in $T$ as:
\begin{equation*}
a_1I_n+a_2T+a_3T^2+ \cdots + a_nT^{n-1}.
\end{equation*}
Since $T$ satisfies $T^n=I_n$, this shows that circulant matrices
commute. This property of circulant matrices is essential in the
four-circulant constructions as well as the constructions that we
will be using in subsequent sections.

$lambda$-circulant matrices are similar to circulant matrices, where
instead of the right circular shift, the $\lambda$-circular shift is
used:
\begin{equation*}
\sigma_{\lambda}(a_1, a_2, \dots, a_n) = (\lambda a_n, a_1, \dots,
a_{n-1}).
\end{equation*}
Thus a $\lambda$-circulant matrix is a matrix of the form

\begin{equation}
\left[
\begin{array}{c}
\overline{r} \\ \hline \sigma_{\lambda}(\overline{r}) \\ \hline
\sigma_{\lambda}^2(\overline{r}) \\ \hline \vdots \\ \hline
\sigma_{\lambda}^{n-1}(\overline{r})%
\end{array}
\right],
\end{equation}
where $\overline{r}$ is the first row. $\lambda$-circulant matrices
share the commutativity property of circulant matrices in matrix
multiplication.

When $\lambda=1$ we get the circulant matrices and when $\lambda=-1$
we get the so-called negacirculant matrices. Negacirculant matrices
have recently been used for constructing self-dual codes in
\cite{Patrick}.
\subsection{Background on Codes}

Let $R$ be a finite ring. A linear \emph{code} $C$ of length $n$ \
over $R$
is an $R$-submodule of $R^{n}$. The elements of $C$ are called \emph{%
codewords}.

Let $\langle \textbf{u},\textbf{v}\rangle $ be inner product of two codewords $\textbf{u}$ and $\textbf{v}$ in $%
R^{n}$ which is defined as $\langle \textbf{u},\textbf{v}\rangle
=\sum_{i=1}^{n}u_{i}v_{i}$, where the operations are done in
$R^{n}$. The \emph{dual code} of a code $C$ is $C^{\perp
}=\{\textbf{v}\in R^{n}\mid \langle \textbf{u},\textbf{v}\rangle =0$
for all $\textbf{v}\in C\}$. If $C\subseteq C^{\perp }$, $C$ is
called \emph{self-orthogonal}$,$ and $C$ is\emph{\ self-dual} if
$C=C^{\perp }$.

The main case of interest for us is the case when $R=\mathbb{F}_2$,
in which case we obtain the usual binary self-dual codes. Binary
self-dual codes are called Type II if the weights of all codewords
are multiples of 4 and Type I otherwise. Rains finalized the upper
bound for the minimum distance $d$ of a binary self-dual code of
length $n$ in \cite{Rains} as $d\leq $ $4\lfloor \frac{n}{24}\rfloor
+6$ if $n\equiv 22\pmod{24}$ and $d\leq $ $4\lfloor
\frac{n}{24}\rfloor +4$, otherwise. A self-dual binary code is
called \textit{extremal} if it meets the bound. Extremal binary
self-dual codes of different lengths have particular weight
enumerators as has been described in \cite{conway},
\cite{dougherty1}. However, while for some lengths a complete
classification has been completed, for many lengths the existence of
codes with a particular weight enumerator is still an open problem.
With the constructions that we apply in this work, we have added to
the list of known codes.

We will be considering two special rings besides the binary field in
constructing our examples, i.e., the ring $\mathbb{F}_2+u\mathbb{F}_2$ and $%
\mathbb{F}_4+u\mathbb{F}_4$. Let
$\mathbb{F}_{4}=\mathbb{F}_{2}\left( \omega \right) $ be the
quadratic field extension of $\mathbb{F}_2$, where $\omega
^{2}+\omega +1=0$. The ring $\mathbb{F}_{4}+u\mathbb{F}_{4}$ defined via $%
u^{2}=0$ is a commutative binary ring of size $16$. We may easily
observe that it is isomorphic to $\mathbb{F}_{2}\left[ \omega
,u\right] /\left\langle u^{2},\omega ^{2}+\omega +1\right\rangle $.
The ring has a unique non-trivial ideal $\left\langle u\right\rangle
=\left\{ 0,u,u\omega ,u+u\omega \right\} $. Note that
$\mathbb{F}_4+u\mathbb{F}_4$ can be viewed as an extension of
$\mathbb{F}_2+u\mathbb{F}_2$ and so we can describe any element of
$\mathbb{F}_4+u\mathbb{F}_4$ in the form $\omega a+\bar{\omega}b$
uniquely, where $a,b \in \mathbb{F}_2+u\mathbb{F}_2$.

The maps $\phi_1: \mathbb{F}_2+u\mathbb{F}_2 \rightarrow
\mathbb{F}_2^2$, given by $\phi_1(a+ub) = (b, a+b)$ and
\begin{equation*}
\varphi _{\mathbb{F}_{4}+u\mathbb{F}_{4}}:\left( \mathbb{F}_{4}+u\mathbb{F}%
_{4}\right) ^{n}\rightarrow \left(
\mathbb{F}_{2}+u\mathbb{F}_{2}\right)
^{2n}, a\omega +b\overline{\omega }\mapsto \left( a,b\right) \text{, \ }%
a,b\in \left( \mathbb{F}_{2}+u\mathbb{F}_{2}\right) ^{n}
\end{equation*}
are orthogonality and distance preserving maps that were described
partially in \cite{ling} and were fully described and used in
\cite{kayayildiz}. They will be used here to construct binary
self-dual codes.

In order to fit the upcoming tables we use hexadecimal number sytem
to describe the elements of $\F_4+u\F_4$. The one-to-one
correspondence between hexadecimals and binary $4$ tuples is as
follows:
\begin{eqnarray*}
0 &\leftrightarrow &0000,\ 1\leftrightarrow 0001,\ 2\leftrightarrow
0010,\
3\leftrightarrow 0011, \\
4 &\leftrightarrow &0100,\ 5\leftrightarrow 0101,\ 6\leftrightarrow
0110,\
7\leftrightarrow 0111, \\
8 &\leftrightarrow &1000,\ 9\leftrightarrow 1001,\ A\leftrightarrow
1010,\
B\leftrightarrow 1011, \\
C &\leftrightarrow &1100,\ D\leftrightarrow 1101,\ E\leftrightarrow
1110,\ F\leftrightarrow 1111.
\end{eqnarray*}
To express elements of $\mathbb{F}_{4}+u\mathbb{F}_{4}$, we use the
ordered
basis $\left\{ u\omega ,\omega ,u,1\right\} $. For instance $1+u\omega $ in $%
\mathbb{F}_{4}+u\mathbb{F}_{4}$ is expressed as $1001$ which is $9$.

\subsection{Arrays for orthogonal designs}
We begin with the following general definition of an orthogonal
design:
\begin{definition}
An orthogonal design of order $n$ and type $\left(
u_{1},u_{2},\ldots ,u_{s}\right) $ on variables $x_{1},x_{2},\ldots
,x_{s}$ is an $n\times n$ matrix $A$ with entries $\left\{ 0,\pm
x_{1},\pm x_{2},\ldots ,\pm
x_{s}\right\} $ where $x_{i}$'s are commuting indeterminates and%
\begin{equation*}
AA^{T}=\sum\limits_{i=1}^{s}u_{i}x_{i}^{2}I_{n}.
\end{equation*}%
It is denoted by $OD\left( n;u_{1},u_{2},\ldots ,u_{s}\right) $.
\end{definition}

For instance
\begin{equation*}
\left(
\begin{array}{cccc}
a & b & c & d \\
-b & a & -d & c \\
-c & d & a & -b \\
-d & -c & b & a%
\end{array}%
\right)
\end{equation*}%
is an $OD\left( 4;1,1,1,1\right) $. When we replace $a,b,c$ and $d$
respectively with symmetric circulant matrices we get the Williamson array \cite%
{williamson}
\begin{equation*}
\left(
\begin{array}{cccc}
A & B & C & D \\
-B & A & -D & C \\
-C & D & A & -B \\
-D & -C & B & A%
\end{array}%
\right).
\end{equation*}%
Here, $A,B,C$ and $D$ are symmetric circulant matrices.

Baumert-Hall arrays (\cite{baumert}) are a generalization of
Williamson arrays. \

\begin{definition}
A $4t\times 4t$ array of $\pm A,\pm B,\pm C,\pm D$ is said to be a
Baumert-Hall array if each indeterminate occurs exactly $t$ times in
each row and column and the distinct rows are formally orthogonal.
\end{definition}

The Goethals-Seidel array is introduced in \cite{goethals} and is a
special Baumert-Hall array defined as:
\begin{equation*}
\left(
\begin{array}{cccc}
A & BR & CR & DR \\
-BR & A & D^{T}R & -C^{T}R \\
-CR & -D^{T}R & A & B^{T}R \\
-DR & C^{T}R & -B^{T}R & A%
\end{array}%
\right),
\end{equation*}%
where $A,B,C$ and $D$ are circulant square matrices and $R$ is the
back diagonal matrix. It is used to construct self-dual codes in
\cite{betsumiya}.

A short Kharaghani array is introduced in \cite{kharaghani};
\begin{equation*}
\left(
\begin{array}{cccc}
A & B & CR & DR \\
-B & A & DR & -CR \\
-CR & -DR & A & B \\
-DR & CR & -B & A%
\end{array}%
\right)
\end{equation*}%
where $A,B,C$ and $D$ are circulant square matrices and $R$ is the
back diagonal matrix. Recently, the array and a variation of it are
used to construct self-dual codes in \cite{shortkharaghani}.

\section{Self-dual codes via Baumert-Hall arrays\label{baumerthall}}

In this section, we give constructions for self-dual codes via a
Baumert-Hall array. The constructions are applicable over
commutative
Frobenius rings with various characteristics. Throughout the section let $%
\mathcal{R}$ denote a commutative Frobenius ring.

In \cite{tez}, Gholamiangonabadi and Kharaghani introduced the
following Baumert-Hall array;

\begin{equation}
\left(
\begin{array}{cccc}
A & B & C & D \\
-B & A & -D & C \\
-C^{T} & D^{T} & A^{T} & -B^{T} \\
-D^{T} & -C^{T} & B^{T} & A^{T}%
\end{array}%
\right)  \label{array}
\end{equation}%
where $A,B,C$ and $D$ are circulant matrices which are amicable with
the pairing $(A,B)$ and $(C,D)$ (i.e., $AB^T=BA^T$ and $CD^T=DC^T$).
The array \ref{array} can be used to construct self-dual codes as
follows:

\begin{theorem}
\label{baumert} Let $\lambda $ be an element of the ring $\mathcal{R}$ with $%
\lambda ^{2}=1$. Let $\mathcal{C}$ be the linear code over
$\mathcal{R}$ of length $8n$ generated by the matrix in the
following form;
\begin{equation*}
G:=\left( \enspace I_{4n}\enspace%
\begin{array}{|cccc}
A & B & C & D \\
-B & A & -D & C \\
-C^{T} & D^{T} & A^{T} & -B^{T} \\
-D^{T} & -C^{T} & B^{T} & A^{T}%
\end{array}%
\right)
\end{equation*}%
where $A,B,C$ and $D$ are $\lambda $-circulant matrices over the ring $%
\mathcal{R}$ satisfying the conditions%
\begin{eqnarray}
AA^{T}+BB^{T}+CC^{T}+DD^{T} &=&-I_{n}\text{ and}  \label{baumerth} \\
AB^{T}-BA^{T}+CD^{T}-DC^{T} &=&0.  \label{baumerth2}
\end{eqnarray}%
Then $\mathcal{C}$ is self-dual.
\end{theorem}

\begin{proof}
Let $A,B,C$ and $D$ be $\lambda$-circulant matrices and
\begin{equation*}
N:=\left(
\begin{array}{cccc}
A & B & C & D \\
-B & A & -D & C \\
-C^{T} & D^{T} & A^{T} & -B^{T} \\
-D^{T} & -C^{T} & B^{T} & A^{T}%
\end{array}%
\right)
\end{equation*}
then it is enough to show that $NN^{T}=-I_{4n}$. We observe that
\begin{equation*}
NN^{T}=\left(
\begin{array}{cccc}
X & Y & Z & T \\
-Y & X & -T & -Z \\
Z^{T} & -T^{T} & X^{T} & U \\
T^{T} & Z^{T} & -U & X^{T}%
\end{array}%
\right)
\end{equation*}
where%
\begin{eqnarray*}
X &=&AA^{T}+BB^{T}+CC^{T}+DD^{T}, \\
Y &=&-AB^{T}+BA^{T}-CD^{T}+DC^{T}, \\
Z &=&-AC+BD+CA-BD, \\
T &=&-AD-BC+CB+DA, \\
U &=&C^{T}D-D^{T}C+A^{T}B-B^{T}A.
\end{eqnarray*}
Circulant matrices commute therefore $Z=0$ and $T=0$. By (\ref{baumerth2}) $%
Y=0$ and $U=0$. By (\ref{baumerth}) $X=-I_n$. Result follows.
\end{proof}

A particular case of Theorem \ref{baumert} is given where the condition (\ref%
{baumerth}) is splitted into two conditions, which allows us to
search for pairs of matrices $(A,B)$ and $(C,D)$ stepwise.

\begin{corollary}
\label{amicable} Let $\lambda $ be an element of the ring
$\mathcal{R} $ with $\lambda ^{2}=1$. Let $\mathcal{C}$ be the
linear code over $\mathcal{R} $ of length $8n$ generated by the
matrix in the following form;
\begin{equation*}
G:=\left( \enspace I_{4n}\enspace%
\begin{array}{|cccc}
A & B & C & D \\
-B & A & -D & C \\
-C^{T} & D^{T} & A^{T} & -B^{T} \\
-D^{T} & -C^{T} & B^{T} & A^{T}%
\end{array}%
\right)
\end{equation*}%
where $A,B,C$ and $D$ are $\lambda $-circulant matrices over the ring $%
\mathcal{R}$ satisfying the conditions%
\begin{eqnarray}
AA^{T}+BB^{T}+CC^{T}+DD^{T} &=&-I_{n},  \label{amicable1} \\
AB^{T}-BA^{T} &=&0\text{ and}  \label{amicable2} \\
CD^{T}-DC^{T} &=&0.  \label{amicable3}
\end{eqnarray}%
Then $\mathcal{C}$ is self-dual.
\end{corollary}

It is easily observed that symmetric circulant matrices are
amicable. Therefore we have the following result;

\begin{corollary}
\label{symmetric} Let $\lambda $ be an element of the ring
$\mathcal{R}$ with $\lambda ^{2}=1$. Let $\mathcal{C}$ be the linear
code over $\mathcal{R} $ of length $8n$ generated by the matrix in
the following form;
\begin{equation*}
G:=\left( \enspace I_{4n}\enspace%
\begin{array}{|cccc}
A & B & C & D \\
-B & A & -D & C \\
-C^{T} & D^{T} & A^{T} & -B^{T} \\
-D^{T} & -C^{T} & B^{T} & A^{T}%
\end{array}%
\right)
\end{equation*}%
where $A,B$ are symmetric circulant matrices and $C$ and $D$ are $\lambda $%
-circulant matrices over the ring $\mathcal{R}$ satisfying the conditions%
\begin{eqnarray}
AA^{T}+BB^{T}+CC^{T}+DD^{T} &=&I_{n}\text{ and}  \label{symmetric1} \\
CD^{T}-DC^{T} &=&0.  \label{symmetric2}
\end{eqnarray}%
Then $\mathcal{C}$ is self-dual.
\end{corollary}

\begin{remark}
Note that if we assume $A,B,C$ and $D$ are all symmetric circulant
matrices, then we obtain a special case of \ Corollary
\ref{symmetric}, which corresponds to Williamson array. Through
computational results we observed that this case is not promising in
constructing self-dual codes.
\end{remark}

\section{Computational Results}

\label{computational} The constructions introduced in Section \ref%
{baumerthall} are applied over rings of characteristic $2$ such as $\mathbb{F%
}_{2}$, $\mathbb{F}_{2}+u\mathbb{F}_{2}$ and $\mathbb{F}_{4}+u\mathbb{F}_{4}$%
. By using the corresponding Gray maps binary self-dual codes of
lengths $64$ and $80$ have been constructed.

The possible weight enumerators of extremal Type I self-dual codes
(of parameters $\left[ 64,32,12\right]$) were determined in
\cite{conway} as:
\begin{eqnarray*}
W_{64,1} &=&1+\left( 1312+16\beta \right) y^{12}+\left(
22016-64\beta
\right) y^{14}+\cdots ,:14\leq \beta \leq 284, \\
W_{64,2} &=&1+\left( 1312+16\beta \right) y^{12}+\left(
23040-64\beta \right) y^{14}+\cdots ,:0\leq \beta \leq 277.
\end{eqnarray*}%
The existence of the codes is unknown for most of the $\beta $
values. Most recently codes with new parameters were constructed in
\cite{kaya} (by a bordered four circulant construction) and
\cite{anev}. Together with these, codes exist with weight
enumerators for $\beta =$14, 16, 18, 20,
22, 24, 25, 26, 28, 29, 30, 32, 36, 39, 44, 46, 53, 59, 60, 64 and 74 in $%
W_{64,1}$ and for $\beta =$0, 1, 2, 3, 4, 5,\ 6, 7, 8, 9, 10, 11,
12, 13,\ 14, 15, 16,\ 17, 18, 20, 21,\ 22, 23, 24,\ 25,\ 26, 27,
28,\ 29,\ 30, 32,\ 33,\ 35, 36, 37, 38, 40,\ 41,\ 44, 48, 51,\ 52,\
56, 58, 64, 72, 80,\ 88,\ 96, 104, 108,\ 112,\ 114,\ 118,\ 120 and
184 in $W_{64,2}$.

In the following tables, we give extremal binary self-dual codes of
length $64$ obtained from applying constructions described in
Theorem 3.1 and Corollary 3.2 over the ring $\F_4+u\F_4$. The
orthogonality-preserving Gray map
$\phi_1\circ\varphi_{\F_4+u\F_4}:(\F_4+u\F_4)^n \rightarrow
\F_2^{4n}$ is used in finding the binary codes. Thus, we need
self-dual codes over $\F_4+u\F_4$ of length 16 for such codes. $r_A,
r_B, r_C, r_D$ denote the first rows of the matrices $A, B, C, D$
that appear in the constructions. \clearpage

\begin{table}[H]
\caption{Extremal binary self-dual codes of length 64 from self-dual
codes of length $16$ over $\mathbb{F}_{4}+u\mathbb{F}_{4}$ by
Theorem \protect\ref{baumert}} \label{tab:BH}
\begin{center}
\begin{tabular}{||c|c|c|c|c|c|c||c||}
\hline
$\mathcal{D}_{i}$ & $\lambda $ & $r_{A}$ & $r_{B}$ & $r_{C}$ & $r_{D}$ & $%
\left\vert Aut(\mathcal{D}_{i})\right\vert $ & $\beta $ in $W_{64,2}$ \\
\hline
$\mathcal{D}_{1}$ & $3$ & $\left( 1,B\right) $ & $\left( 7,C\right) $ & $%
\left( 6,D\right) $ & $\left( 4,5\right) $ & $2^{4}$ & $4$ \\ \hline
$\mathcal{D}_{2}$ & $3$ & $\left( C,6\right) $ & $\left( D,9\right) $ & $%
\left( 2,3\right) $ & $\left( C,0\right) $ & $2^{4}$ & $12$ \\
\hline
$\mathcal{D}_{3}$ & $3$ & $\left( 6,E\right) $ & $\left( A,9\right) $ & $%
\left( 0,E\right) $ & $\left( F,B\right) $ & $2^{5}$ & $24$ \\
\hline
$\mathcal{D}_{4}$ & $3$ & $\left( F,8\right) $ & $\left( E,9\right) $ & $%
\left( C,1\right) $ & $\left( A,4\right) $ & $2^{4}$ & $28$ \\
\hline
$\mathcal{D}_{5}$ & $3$ & $\left( 6,C\right) $ & $\left( 1,D\right) $ & $%
\left( 1,4\right) $ & $\left( F,D\right) $ & $2^{5}$ & $32$ \\
\hline
$\mathcal{D}_{6}$ & $3$ & $\left( C,4\right) $ & $\left( 9,5\right) $ & $%
\left( F,5\right) $ & $\left( E,9\right) $ & $2^{5}$ & $36$ \\
\hline
$\mathcal{D}_{7}$ & $3$ & $\left( 5,5\right) $ & $\left( 9,3\right) $ & $%
\left( 2,5\right) $ & $\left( 7,1\right) $ & $2^{5}$ & $40$ \\
\hline
$\mathcal{D}_{8}$ & $3$ & $\left( D,8\right) $ & $\left( 0,3\right) $ & $%
\left( 4,1\right) $ & $\left( F,5\right) $ & $2^{4}\times 3$ & $44$
\\ \hline
$\mathcal{D}_{9}$ & $3$ & $\left( E,E\right) $ & $\left( 1,5\right) $ & $%
\left( 5,5\right) $ & $\left( C,9\right) $ & $2^{5}$ & $48$ \\
\hline
$\mathcal{D}_{10}$ & $3$ & $\left( E,E\right) $ & $\left( 1,7\right) $ & $%
\left( 5,5\right) $ & $\left( 4,9\right) $ & $2^{5}$ & $52$ \\
\hline\hline
$\mathcal{D}_{11}$ & $9$ & $\left( 1,F\right) $ & $\left( 6,C\right) $ & $%
\left( 1,A\right) $ & $\left( C,2\right) $ & $2^{3}$ & $0$ \\ \hline
$\mathcal{D}_{12}$ & $9$ & $\left( D,2\right) $ & $\left( 1,4\right) $ & $%
\left( 7,F\right) $ & $\left( B,8\right) $ & $2^{3}$ & $4$ \\ \hline
$\mathcal{D}_{13}$ & $9$ & $\left( 9,C\right) $ & $\left( 3,3\right) $ & $%
\left( 5,B\right) $ & $\left( 2,2\right) $ & $2^{3}$ & $5$ \\ \hline
$\mathcal{D}_{14}$ & $9$ & $\left( B,3\right) $ & $\left( 6,E\right) $ & $%
\left( 4,9\right) $ & $\left( E,2\right) $ & $2^{4}$ & $8$ \\ \hline
$\mathcal{D}_{15}$ & $9$ & $\left( 1,7\right) $ & $\left( 6,9\right) $ & $%
\left( 5,D\right) $ & $\left( 6,4\right) $ & $2^{3}$ & $12$ \\
\hline
$\mathcal{D}_{16}$ & $9$ & $\left( B,A\right) $ & $\left( 7,B\right) $ & $%
\left( 6,4\right) $ & $\left( C,0\right) $ & $2^{3}$ & $13$ \\
\hline
$\mathcal{D}_{17}$ & $9$ & $\left( B,4\right) $ & $\left( D,D\right) $ & $%
\left( 4,C\right) $ & $\left( B,D\right) $ & $2^{3}$ & $16$ \\
\hline
$\mathcal{D}_{18}$ & $9$ & $\left( 0,2\right) $ & $\left( C,9\right) $ & $%
\left( 9,3\right) $ & $\left( B,7\right) $ & $2^{3}$ & $17$ \\
\hline
$\mathcal{D}_{19}$ & $9$ & $\left( 3,D\right) $ & $\left( A,F\right) $ & $%
\left( 8,6\right) $ & $\left( D,1\right) $ & $2^{3}$ & $21$ \\
\hline
$\mathcal{D}_{20}$ & $9$ & $\left( 7,3\right) $ & $\left( D,4\right) $ & $%
\left( 2,9\right) $ & $\left( 1,6\right) $ & $2^{4}$ & $32$ \\
\hline\hline
$\mathcal{D}_{21}$ & $B$ & $\left( C,3\right) $ & $\left( 2,2\right) $ & $%
\left( B,D\right) $ & $\left( 3,1\right) $ & $2^{3}$ & $5$ \\ \hline
$\mathcal{D}_{22}$ & $B$ & $\left( B,4\right) $ & $\left( 9,F\right) $ & $%
\left( E,C\right) $ & $\left( D,D\right) $ & $2^{4}$ & $8$ \\ \hline
$\mathcal{D}_{23}$ & $B$ & $\left( E,9\right) $ & $\left( 8,2\right) $ & $%
\left( B,F\right) $ & $\left( 1,1\right) $ & $2^{3}$ & $9$ \\ \hline
$\mathcal{D}_{24}$ & $B$ & $\left( 5,B\right) $ & $\left( 5,2\right) $ & $%
\left( 6,8\right) $ & $\left( 3,5\right) $ & $2^{3}$ & $13$ \\
\hline
$\mathcal{D}_{25}$ & $B$ & $\left( 3,6\right) $ & $\left( 8,6\right) $ & $%
\left( A,F\right) $ & $\left( C,9\right) $ & $2^{3}$ & $17$ \\
\hline
$\mathcal{D}_{26}$ & $B$ & $\left( C,4\right) $ & $\left( F,D\right) $ & $%
\left( 9,C\right) $ & $\left( F,B\right) $ & $2^{3}$ & $20$ \\
\hline
$\mathcal{D}_{27}$ & $B$ & $\left( 6,9\right) $ & $\left( A,D\right) $ & $%
\left( 8,6\right) $ & $\left( 1,6\right) $ & $2^{3}$ & $21$ \\
\hline
$\mathcal{D}_{28}$ & $B$ & $\left( 7,1\right) $ & $\left( D,8\right) $ & $%
\left( C,8\right) $ & $\left( 9,7\right) $ & $2^{3}$ & $25$ \\
\hline
$\mathcal{D}_{29}$ & $B$ & $\left( B,5\right) $ & $\left( 5,D\right) $ & $%
\left( B,3\right) $ & $\left( 0,7\right) $ & $2^{3}$ & $28$ \\
\hline
$\mathcal{D}_{30}$ & $B$ & $\left( 3,3\right) $ & $\left( 6,6\right) $ & $%
\left( 3,E\right) $ & $\left( A,4\right) $ & $2^{4}$ & $40$ \\
\hline
\end{tabular}%
\end{center}
\end{table}

\begin{table}[H]
\caption{Extremal binary self-dual codes of length 64 from self-dual
codes of length $16$ over $\mathbb{F}_{4}+u\mathbb{F}_{4}$ by
Corollary \protect\ref{amicable}} \label{tab:amicable}
\begin{center}
\begin{tabular}{||c|c|c|c|c|c|c||c||}
\hline
$\mathcal{E}_{i}$ & $\lambda $ & $r_{A}$ & $r_{B}$ & $r_{C}$ & $r_{D}$ & $%
\left\vert Aut(\mathcal{E}_{i})\right\vert $ & $\beta $ in $W_{64,2}$ \\
\hline\hline
$\mathcal{E}_{1}$ & $3$ & $\left( F,B\right) $ & $\left( 1,E\right) $ & $%
\left( 7,1\right) $ & $\left( 0,A\right) $ & $2^{4}$ & $0$ \\ \hline
$\mathcal{E}_{2}$ & $3$ & \texttt{"} & \texttt{"} & $\left( 9,9\right) $ & $%
\left( A,A\right) $ & $2^{5}$ & $4$ \\ \hline
$\mathcal{E}_{3}$ & $3$ & \texttt{"} & \texttt{"} & $\left( 7,D\right) $ & $%
\left( E,4\right) $ & $2^{4}$ & $24$ \\ \hline
$\mathcal{E}_{4}$ & $3$ & \texttt{"} & \texttt{"} & $\left( 1,1\right) $ & $%
\left( 8,8\right) $ & $2^{5}$ & $28$ \\ \hline
$\mathcal{E}_{5}$ & $3$ & \texttt{"} & \texttt{"} & $\left( 4,4\right) $ & $%
\left( 7,7\right) $ & $2^{5}$ & $52$ \\ \hline
$\mathcal{E}_{6}$ & $3$ & $\left( 9,6\right) $ & $\left( F,9\right) $ & $%
\left( 5,F\right) $ & $\left( 4,C\right) $ & $2^{5}$ & $12$ \\
\hline
$\mathcal{E}_{7}$ & $3$ & \texttt{"} & \texttt{"} & $\left( F,7\right) $ & $%
\left( C,C\right) $ & $2^{4}$ & $16$ \\ \hline
$\mathcal{E}_{8}$ & $3$ & \texttt{"} & \texttt{"} & $\left( E,E\right) $ & $%
\left( 7,7\right) $ & $2^{5}$ & $20$ \\ \hline
$\mathcal{E}_{9}$ & $3$ & \texttt{"} & \texttt{"} & $\left( F,F\right) $ & $%
\left( 4,4\right) $ & $2^{5}$ & $36$ \\ \hline\hline
$\mathcal{E}_{10}$ & $9$ & $\left( F,B\right) $ & $\left( B,C\right) $ & $%
\left( 2,2\right) $ & $\left( 3,9\right) $ & $2^{3}$ & $5$ \\ \hline
$\mathcal{E}_{11}$ & $9$ & $\left( D,3\right) $ & $\left( 3,4\right) $ & $%
\left( A,2\right) $ & $\left( 3,B\right) $ & $2^{3}$ & $13$ \\
\hline
$\mathcal{E}_{12}$ & $9$ & $\left( 0,2\right) $ & $\left( B,3\right) $ & $%
\left( B,D\right) $ & $\left( CB\right) $ & $2^{3}$ & $17$ \\
\hline\hline
$\mathcal{E}_{13}$ & $B$ & $\left( 3,1\right) $ & $\left( 2,2\right) $ & $%
\left( B,5\right) $ & $\left( C,9\right) $ & $2^{3}$ & $5$ \\ \hline
$\mathcal{E}_{14}$ & $B$ & $\left( 9,6\right) $ & $\left( 5,9\right) $ & $%
\left( 0,2\right) $ & $\left( B,9\right) $ & $2^{3}$ & $17$ \\
\hline
\end{tabular}%
\end{center}
\end{table}

\subsection{New Type II self-dual $\left[ 72,36,12\right] _{2}$-codes}
The existence of Type II extremal binary code of length $72$ is
still an open problem. The best known Type II self-dual codes of
length $72$ have parameters $[72,36,12]$. The possible weight
enumerators for these codes are given in \cite{dougherty1} as
\begin{equation*}
\begin{tabular}{l}
$W_{72}=1+(4398+\alpha )y^{12}+(197073-12\alpha
)y^{16}+$\textperiodcentered \textperiodcentered \textperiodcentered
\end{tabular}%
\end{equation*}%
Note that a code with weight enumerator $\alpha =-4398$ would
correspond to the above-mentioned extremal binary self-dual code (of
parameters $\left[ 72,36,16\right]$). For a list of known $\alpha $
values we refer to \cite{tufekci}. We construct codes with new $\alpha$ values by Corollary \ref%
{symmetric} over the binary field $\F_2$. Since $A$ and $B$ are symmetric circulant matrices only first $%
5 $ entries of the first rows are given in Table \ref{tab:72sym}.

\begin{table}[H]
\caption{Type II self-dual codes of length 72 by Corollary \protect\ref%
{symmetric}} \label{tab:72sym}
\begin{center}
\begin{tabular}{||c|c|c|c|c||c||}
\hline
$\mathcal{C}_{72,i}$ & $r_{A}$ & $r_{B}$ & $r_{C}$ & $r_{D}$ & $\alpha $ \\
\hline\hline
$\mathcal{C}_{72,1}$ & $\left( 11011\right) $ & $\left( 01010\right) $ & $%
\left( 101100101\right) $ & $\left( 110001000\right) $ & $-2736$ \\
\hline
$\mathcal{C}_{72,2}$ & $\left( 10110\right) $ & $\left( 00000\right) $ & $%
\left( 001110110\right) $ & $\left( 010011110\right) $ & $-2748$ \\
\hline
$\mathcal{C}_{72,3}$ & $\left( 11001\right) $ & $\left( 10101\right) $ & $%
\left( 011010010\right) $ & $\left( 001001111\right) $ & $-2844$ \\
\hline
$\mathcal{C}_{72,4}$ & $\left( 00111\right) $ & $\left( 10110\right) $ & $%
\left( 000100011\right) $ & $\left( 110010110\right) $ & $-2964$ \\
\hline
$\mathcal{C}_{72,5}$ & $\left( 11001\right) $ & $\left( 10101\right) $ & $%
\left( 001100010\right) $ & $\left( 011011011\right) $ & $-3060$ \\
\hline
$\mathcal{C}_{72,6}$ & $\left( 00000\right) $ & $\left( 10100\right) $ & $%
\left( 010111010\right) $ & $\left( 000100101\right) $ & $-3396$ \\
\hline
\end{tabular}%
\end{center}
\end{table}

\begin{example}
Let $A=circ\left( 001101000\right) ,$ $B=circ\left( 100100100\right) ,$ $%
C=circ\left( 000101001\right) $ and $B=circ\left( 111111000\right) $
which are amicable with the pairing $\left( A,B\right) $ and $\left(
C,D\right) $. In other words, the conditions \ref{amicable2} and
\ref{amicable3} are
satisfied. Moreover, they satisfy the equation \ref{amicable1}. Let $%
\mathcal{C}_{72,7}$ be the code obtained via Corollary
\ref{amicable}. Then it is a Type II self-dual $\left[
72,36,12\right]$ code with weight enumerator $\alpha =-3618$ and
automorphism group of order $36$.
\end{example}

By using Theorem \ref{baumert} we obtain more codes, which are
listed in Table \ref{tab:72bh}.

\begin{table}[H]
\caption{Type II self-dual codes of length 72 by Theorem \protect\ref%
{baumert}} \label{tab:72bh}
\begin{center}
\begin{tabular}{||c|c|c|c|c||c||}
\hline
$\mathcal{C}_{72,i}$ & $r_{A}$ & $r_{B}$ & $r_{C}$ & $r_{D}$ & $\alpha $ \\
\hline\hline $\mathcal{C}_{72,8}$ & $\left( 011111011\right) $ &
$\left( 001100111\right)
$ & $\left( 110110011\right) $ & $\left( 010011101\right) $ & $-2682$ \\
\hline $\mathcal{C}_{72,9}$ & $\left( 000011111\right) $ & $\left(
100001110\right)
$ & $\left( 111011011\right) $ & $\left( 010000110\right) $ & $-2700$ \\
\hline $\mathcal{C}_{72,10}$ & $\left( 011011111\right) $ & $\left(
011100100\right) $ & $\left( 000010110\right) $ & $\left(
101010011\right) $ & $-2754$ \\ \hline $\mathcal{C}_{72,11}$ &
$\left( 010011101\right) $ & $\left( 111000011\right) $ & $\left(
010000101\right) $ & $\left( 111100101\right) $ & $-2790$ \\ \hline
$\mathcal{C}_{72,12}$ & $\left( 110111001\right) $ & $\left(
110100101\right) $ & $\left( 100101011\right) $ & $\left(
000011100\right) $ & $-2802$ \\ \hline $\mathcal{C}_{72,13}$ &
$\left( 011000100\right) $ & $\left( 101110111\right) $ & $\left(
101010100\right) $ & $\left( 110110100\right) $ & $-2862$ \\ \hline
$\mathcal{C}_{72,14}$ & $\left( 100111011\right) $ & $\left(
010011110\right) $ & $\left( 001110000\right) $ & $\left(
000110111\right) $ & $-2982$ \\ \hline $\mathcal{C}_{72,15}$ &
$\left( 100101000\right) $ & $\left( 111111001\right) $ & $\left(
000011110\right) $ & $\left( 011110010\right) $ & $-2988$ \\ \hline
$\mathcal{C}_{72,16}$ & $\left( 110001000\right) $ & $\left(
011011001\right) $ & $\left( 011111010\right) $ & $\left(
010110101\right) $ & $-3132$ \\ \hline $\mathcal{C}_{72,17}$ &
$\left( 001000011\right) $ & $\left( 000111000\right) $ & $\left(
001101110\right) $ & $\left( 001110010\right) $ & $-3150$ \\ \hline
$\mathcal{C}_{72,18}$ & $\left( 000010110\right) $ & $\left(
001101111\right) $ & $\left( 010101110\right) $ & $\left(
111111111\right) $ & $-3654$ \\ \hline $\mathcal{C}_{72,19}$ &
$\left( 101111011\right) $ & $\left( 110101101\right) $ & $\left(
000011111\right) $ & $\left( 001101011\right) $ & $-3690$ \\ \hline
$\mathcal{C}_{72,20}$ & $\left( 000001111\right) $ & $\left(
101101111\right) $ & $\left( 000111000\right) $ & $\left(
101100110\right) $ & $-3774$ \\ \hline $\mathcal{C}_{72,21}$ &
$\left( 000100000\right) $ & $\left( 010010111\right) $ & $\left(
001111100\right) $ & $\left( 000111010\right) $ & $-3780$ \\ \hline
$\mathcal{C}_{72,22}$ & $\left( 100101101\right) $ & $\left(
000110110\right) $ & $\left( 010000101\right) $ & $\left(
111000000\right) $ & $-3792$ \\ \hline $\mathcal{C}_{72,23}$ &
$\left( 000010101\right) $ & $\left( 011110001\right) $ & $\left(
011010111\right) $ & $\left( 000111110\right) $ & $-3906$ \\ \hline
$\mathcal{C}_{72,24}$ & $\left( 010000010\right) $ & $\left(
000001110\right) $ & $\left( 110100101\right) $ & $\left(
011010011\right) $ & $-3918$ \\ \hline $\mathcal{C}_{72,25}$ &
$\left( 101101111\right) $ & $\left( 101110110\right) $ & $\left(
001011000\right) $ & $\left( 000100110\right) $ & $-4068$ \\ \hline
$\mathcal{C}_{72,26}$ & $\left( 100010111\right) $ & $\left(
001100010\right) $ & $\left( 100110001\right) $ & $\left(
000101001\right) $ & $-4086$ \\ \hline
\end{tabular}%
\end{center}
\end{table}

\subsection{New extremal Type II binary self-dual codes of length $80$}

The weight enumerator of an extremal Type II binary self-dual code
of length $80$ ( of parameters $\left[ 80,40,16\right]$) is
uniquely determined as $1+97565y^{16}+12882688y^{20}+\cdots $ \cite%
{dougherty1}. Recently, new such codes were constructed via the four circulant construction over $\mathbb{F}_{2}+u\mathbb{F}%
_{2}$ in \cite{kayayildiz}. Here, we construct $8$ new codes by
using Theorem \ref{baumert} over the binary alphabet.
The nonequivalence of the codes is checked by the invariants. Let $%
\mathbf{c}_{1},\mathbf{c}_{2},\ldots, \mathbf{c}_{97565}$ be the
codewords of weight $16$ in an
extremal Type II self-dual code of length 80 and let $%
I_{j}=\left\vert \left\{ \left( \mathbf{c}_{k},\mathbf{c}_{l}\right)
|\ d\left( \mathbf{c}_{k}, \mathbf{c}_{l}\right) =j,\ k<l\right\}
\right\vert $ where $d$ is the Hamming distance. Two codes are
inequivalent if their $I_{16}$-values are different since $I_{16}$
is invariant under a permutation of the coordinates.

\begin{table}[H]
\caption{Type II self-dual codes of length 80 by Theorem \protect\ref%
{baumert}} \label{tab:80}
\begin{center}
\resizebox{\textwidth}{!}{
\begin{tabular}{||c|c|c|c|c|c||c||}
\hline $\mathcal{G}_{i}$ & $r_{A}$ & $r_{B}$ & $r_{C}$ & $r_{D}$ &
$\left\vert Aut(\mathcal{G}_{i})\right\vert $ & $I_{16}$ \\
\hline\hline $\mathcal{G}_{1}$ & $\left( 1110110000\right) $ &
$\left( 1101000001\right) $ & $\left( 1001111100\right) $ & $\left(
1001010001\right) $ & $2^{3}\times 5$ & $20039280$ \\ \hline
$\mathcal{G}_{2}$ & $\left( 1100101000\right) $ & $\left(
0100111110\right) $ & $\left( 1101000001\right) $ & $\left(
1110110000\right) $ & $2^{3}\times 5$ & $20248440$ \\ \hline
$\mathcal{G}_{3}$ & $\left( 0000100010\right) $ & $\left(
1001000000\right) $ & $\left( 1000001100\right) $ & $\left(
0111111110\right) $ & $2^{3}\times 5$ & $20306280$ \\ \hline
$\mathcal{G}_{4}$ & $\left( 0100000000\right) $ & $\left(
1000100011\right) $ & $\left( 1001100010\right) $ & $\left(
1111000011\right) $ & $2^{3}\times 5$ & $20342040$ \\ \hline
$\mathcal{G}_{5}$ & $\left( 0111110110\right) $ & $\left(
0101101000\right) $ & $\left( 1011010001\right) $ & $\left(
0001010100\right) $ & $2^{3}\times 5$ & $20457960$ \\ \hline
$\mathcal{G}_{6}$ & $\left( 1111000101\right) $ & $\left(
0100000101\right) $ & $\left( 1111001110\right) $ & $\left(
0110011111\right) $ & $2^{4}\times 3\times 5$ & $19992780$ \\ \hline
$\mathcal{G}_{7}$ & $\left( 1100111110\right) $ & $\left(
1101000001\right) $ & $\left( 0001100110\right) $ & $\left(
0100100110\right) $ & $2^{4}\times 3\times 5$ & $20008440$ \\ \hline
$\mathcal{G}_{8}$ & $\left( 1001100100\right) $ & $\left(
1001100001\right) $ & $\left( 1101011011\right) $ & $\left(
0010100011\right) $ & $2^{4}\times 3\times 5$ & $20082720$ \\ \hline
\end{tabular}}
\end{center}
\end{table}

Combining this with the last known number of such codes from
\cite{kayayildiz}, we obtain the following theorem:
\begin{theorem}
There exist at least $44$ inequivalent extremal Type II self-dual
codes of length $80$.
\end{theorem}

The codewords of weight $16$ in an extremal Type II code of length
$80$ form a $3$-design by Assmus-Matson theorem. Hence, we have the
following subsequent result:

\begin{theorem}
There are at least $44$ non-isomorphic $3-\left( 80,16,665\right) $
designs.
\end{theorem}

\subsection{New extremal binary self-dual codes of length 68 from $\mathbb{F}%
_{2}+u\mathbb{F}_{2}$-extensions}

In \cite{buyuklieva}, possible weight enumerators of an extremal
binary self-dual code of length $68$ (of parameters
$\left[68,34,12\right]$) are characterized as follows:
\begin{eqnarray*}
W_{68,1} &=&1+\left( 442+4\beta \right) y^{12}+\left( 10864-8\beta
\right)
y^{14}+\cdots , \\
W_{68,2} &=&1+\left( 442+4\beta \right) y^{12}+\left( 14960-8\beta
-256\gamma \right) y^{14}+\cdots ,
\end{eqnarray*}%
where $0\leq \gamma \leq 9$ by \cite{harada}. The existence of codes
is known for various parameters for $W_{68,1}$; for a list of known
such codes we
refer to \cite{tufekci}. The existence of codes is known for $%
\gamma=0,1,2,3,4,5,6$ and $\gamma=7$ in $W_{68,2}$. Recently, the
first examples of codes with $\gamma=5$ were constructed in
\cite{gildea}. Yankov et al. constructed codes with $\gamma=7$ by
considering codes with an automorphism group of order $7$ in
\cite{yankov18}.

We construct extremal binary self-dual codes of length $68$ with new
weight enumerators via the building-up construction over
$\mathbb{F}_{2}+u\mathbb{F}_{2}$ applied to the codes of length 64
obtained at the beginning of Section 4. We obtain examples of
extremal binary self-dual codes of length $68$ with weight
enumerator $\gamma =5$ in $W_{68,2}$.

In the sequel, let $R$ be a commutative ring of characteristic $2$
with identity.

\begin{theorem}
\label{extensionA}$($\cite{frobenius}$)$ Let $\mathcal{C}$ be a
self-dual code over $R$ of length $n$ and $G=(r_{i})$ be a $k\times
n$ generator matrix for $C$, where $r_{i}$ is the $i$-th row of $G$,
$1\leq i\leq k$. Let $c$ be a unit in $R$ such that $c^{2}=1$ and
$X$ be a vector in $R^{n}$ with $\left\langle X,X\right\rangle =1$.
Let $y_{i}=\left\langle
r_{i},X\right\rangle $ for $1\leq i\leq k$. Then the following matrix%
\begin{equation*}
\left[
\begin{array}{cc|c}
1 & 0 & X \\ \hline
y_{1} & cy_{1} & r_{1} \\
\vdots & \vdots & \vdots \\
y_{k} & cy_{k} & r_{k}%
\end{array}%
\right] ,
\end{equation*}%
generates a self-dual code $\mathcal{D}$ over $R$ of length $n+2$.
\end{theorem}

Currently, the existence of codes with weight enumerators for
$\gamma =0,1,2,3,4,6$ and $7$ is known. Recently, new codes in
$W_{68,2}$ have been
obtained in \cite{yankov18,gildea}. These codes exist for $\gamma =3,4$ and $%
5$ in $W_{68,2}$ when

$%
\begin{array}{l}
\gamma =3,\ \beta \in \{2m+1|m=38,40,43,44,47,\dots ,77,79,80,81,83,89,96\}%
\text{ or} \\
\beta \in \{2m|m=39,\dots ,92,94,95,97,98,101,102\}; \\
\gamma =4,\ \beta
=103,105,107,113,115,117,119,121,129,139,141,143,145,149,157,161\
\text{or}
\\
\beta \in \{2m|m=43,46,47,48,49,51,52,54,55,56,58,60,\dots
,90,92,97,98\};
\\
\gamma =5\ \text{with}\ \beta \in \{m|m=158,\ldots ,169\}%
\end{array}%
$

We obtain $46$ new codes with weight enumerators for $\gamma =3$ and
$\beta
=91$ ; $\gamma =4$ and $\beta =$90, 106, 109, 112, 114; $\gamma =5$ and $%
\beta =$113, 116,...,153 in $W_{68,2}$ using Theorem
\ref{extensionA} and neighboring constructions.

The following table contains the new extremal binary self-dual codes
of length $68$ obtained from applying Theorem \ref{extensionA} for
$R = \F_2+u\F_2$ over $\varphi_{\F_4+u\F_4}(\mathcal{D}_{16})$:

\begin{table}[H]
\caption{New extremal binary self-dual codes in $W_{68,2}$ as extensions of $%
\mathcal{D}_{16}$ (5 codes)} \label{tab:68table1}
\begin{center}
\begin{tabular}{||c|c|c||c||c||}
\hline $\mathcal{C}_{68,i}$ &   $c$ & $X$ & $\gamma $ & $\beta $ \\
\hline $\mathcal{C}_{68,1}$ &  $1+u$ & $\left(
1131113u3u0110103uuuuu03u03u1031\right) $ & $3$ & $91$ \\ \hline
$\mathcal{C}_{68,2}$  & $1+u$ & $\left(
30u3001113111113313100u11uu10011\right) $ & $4$ & $90$ \\ \hline
$\mathcal{C}_{68,3}$ &  $1+u$ & $\left(
3131331u3u013030300u0u030u103013\right) $ & $5$ & $119$ \\ \hline
$\mathcal{C}_{68,4}$ &  $1+u$ & $\left(
3131331u3u013030300u0u030u103013\right) $ & $5$ & $129$ \\ \hline
$\mathcal{C}_{68,5}$ &  $1+u$ & $\left(
3111333u3u03303u30uuu001003u3031\right) $ & $5$ & $137$ \\ \hline
\end{tabular}%
\end{center}
\end{table}

\subsection{New extremal binary self-dual codes of length 68 via neighboring construction}

Two binary self-dual codes of length $2k$ are said to be neighbors
if their intersection has dimension $k-1$. Let $C$ be a binary
self-dual code of length $2k$ and $x\in {\mathbb{F}}_{2}^{2k}-C$.
Then $D=\left\langle \left\langle x\right\rangle ^{\bot }\cap
C,x\right\rangle $ is a neighbor of $C$. We consider the neighbors
of the codes in Table \ref{tab:68table1} \ and
obtain new codes with $\gamma =4$ and $5$ which are listed in Table \ref%
{tab:68table3} and \ref{tab:68table2}, respectively. The generator
matrix of $C$ is formed into standard form which allows us to fix
first $34$ entries of $x$ as $0$ without loss of generality. The
remaining $34$ entries of $x$ are given in the corresponding tables.
\clearpage

\begin{table}[H]
\caption{New extremal binary self-dual codes of length $68$ with
$\protect\gamma =5$ (36 codes)} \label{tab:68table2}
\begin{center}
\begin{tabular}{||c|c|c||c||}
\hline $\mathcal{N}_{68,i}$ & $\mathcal{C}_{68,i}$ & $x$ & $\beta $
\\ \hline $\mathcal{N}_{68,1}$ & $\mathcal{C}_{68,3}$ & $\left(
0111010001010001111010111111011111\right) $ & 113 \\ \hline
$\mathcal{N}_{68,2}$ & $\mathcal{C}_{68,2}$ & $\left(
1101100010000010001110010101100001\right) $ & 116 \\ \hline
$\mathcal{N}_{68,3}$ & $\mathcal{C}_{68,3}$ & $\left(
1011100111101100111110001010011010\right) $ & 117 \\ \hline
$\mathcal{N}_{68,4}$ & $\mathcal{C}_{68,3}$ & $\left(
0110001010110000000010011000100101\right) $ & 118 \\ \hline
$\mathcal{N}_{68,5}$ & $\mathcal{C}_{68,3}$ & $\left(
0111000001111000010100000110010011\right) $ & 120 \\ \hline
$\mathcal{N}_{68,6}$ & $\mathcal{C}_{68,3}$ & $\left(
1000100011100010010100111011100011\right) $ & 121 \\ \hline
$\mathcal{N}_{68,7}$ & $\mathcal{C}_{68,4}$ & $\left(
1100110100011001110000100111100001\right) $ & 122 \\ \hline
$\mathcal{N}_{68,8}$ & $\mathcal{C}_{68,4}$ & $\left(
0001010111110011001110000100101010\right) $ & 123 \\ \hline
$\mathcal{N}_{68,9}$ & $\mathcal{C}_{68,3}$ & $\left(
0001101011000101010111000010000011\right) $ & 124 \\ \hline
$\mathcal{N}_{68,10}$ & $\mathcal{C}_{68,3}$ & $\left(
1111001110010000001111100101010011\right) $ & 125 \\ \hline
$\mathcal{N}_{68,11}$ & $\mathcal{C}_{68,3}$ & $\left(
1010001010101011110010110010001001\right) $ & 126 \\ \hline
$\mathcal{N}_{68,12}$ & $\mathcal{C}_{68,3}$ & $\left(
0000111001101111101111110011101101\right) $ & 127 \\ \hline
$\mathcal{N}_{68,13}$ & $\mathcal{C}_{68,3}$ & $\left(
1101011011011001011000101110111111\right) $ & 128 \\ \hline
$\mathcal{N}_{68,14}$ & $\mathcal{C}_{68,3}$ & $\left(
0011000010101011110000000010000011\right) $ & 130 \\ \hline
$\mathcal{N}_{68,15}$ & $\mathcal{C}_{68,4}$ & $\left(
1111100111001100010111110000111010\right) $ & 131 \\ \hline
$\mathcal{N}_{68,16}$ & $\mathcal{C}_{68,5}$ & $\left(
1001011001100011100100000000111100\right) $ & 132 \\ \hline
$\mathcal{N}_{68,17}$ & $\mathcal{C}_{68,3}$ & $\left(
1001011111011011100010001111100101\right) $ & 133 \\ \hline
$\mathcal{N}_{68,18}$ & $\mathcal{C}_{68,4}$ & $\left(
1110110100001001010111000011101000\right) $ & 134 \\ \hline
$\mathcal{N}_{68,19}$ & $\mathcal{C}_{68,4}$ & $\left(
0010011100110101100100001111000000\right) $ & 135 \\ \hline
$\mathcal{N}_{68,20}$ & $\mathcal{C}_{68,4}$ & $\left(
1101010010010000110000110000110101\right) $ & 136 \\ \hline
$\mathcal{N}_{68,21}$ & $\mathcal{C}_{68,5}$ & $\left(
0101100110111011001111011110100010\right) $ & 138 \\ \hline
$\mathcal{N}_{68,22}$ & $\mathcal{C}_{68,5}$ & $\left(
1000100110111001011111101100110110\right) $ & 139 \\ \hline
$\mathcal{N}_{68,23}$ & $\mathcal{C}_{68,3}$ & $\left(
0010100001010011000011111001111111\right) $ & 140 \\ \hline
$\mathcal{N}_{68,24}$ & $\mathcal{C}_{68,5}$ & $\left(
0011110011111010011001011100101011\right) $ & 141 \\ \hline
$\mathcal{N}_{68,25}$ & $\mathcal{C}_{68,5}$ & $\left(
1010101000100011110111000100001110\right) $ & 142 \\ \hline
$\mathcal{N}_{68,26}$ & $\mathcal{C}_{68,4}$ & $\left(
0010110101000101000101111111100011\right) $ & 143 \\ \hline
$\mathcal{N}_{68,27}$ & $\mathcal{C}_{68,4}$ & $\left(
1110111100100100111011100100110110\right) $ & 144 \\ \hline
$\mathcal{N}_{68,28}$ & $\mathcal{C}_{68,5}$ & $\left(
0011010101101001011101010100100001\right) $ & 145 \\ \hline
$\mathcal{N}_{68,29}$ & $\mathcal{C}_{68,5}$ & $\left(
1010111001011110111110010011100010\right) $ & 146 \\ \hline
$\mathcal{N}_{68,30}$ & $\mathcal{C}_{68,5}$ & $\left(
1001011001101010010001110001001011\right) $ & 147 \\ \hline
$\mathcal{N}_{68,31}$ & $\mathcal{C}_{68,5}$ & $\left(
1011001011101110011101100011101100\right) $ & 148 \\ \hline
$\mathcal{N}_{68,32}$ & $\mathcal{C}_{68,5}$ & $\left(
0010111111000000111101001111111001\right) $ & 149 \\ \hline
$\mathcal{N}_{68,33}$ & $\mathcal{C}_{68,5}$ & $\left(
1000110101111111111001100010011111\right) $ & 150 \\ \hline
$\mathcal{N}_{68,34}$ & $\mathcal{C}_{68,5}$ & $\left(
0010010000111010111010010001000000\right) $ & 151 \\ \hline
$\mathcal{N}_{68,35}$ & $\mathcal{C}_{68,5}$ & $\left(
1010010001000111011010000111101001\right) $ & 152 \\ \hline
$\mathcal{N}_{68,36}$ & $\mathcal{C}_{68,5}$ & $\left(
1011010110010100101110101110001100\right) $ & 153 \\ \hline
\end{tabular}%
\end{center}
\end{table}

\begin{table}[H]
\caption{New extremal binary self-dual codes of length 68 with
$\protect\gamma =4$ (5 codes)} \label{tab:68table3}
\begin{center}
\begin{tabular}{||c|c|c||c||}
\hline $\mathcal{N}_{68,i}$ & $\mathcal{C}_{68,i}$ & $x$ & $\beta $
\\ \hline $\mathcal{N}_{68,37}$ & $\mathcal{C}_{68,4}$ & $\left(
0011001111001110101000101111110110\right) $ & 106 \\ \hline
$\mathcal{N}_{68,38}$ & $\mathcal{C}_{68,2}$ & $\left(
1110011010010101100110101000100100\right) $ & 107 \\ \hline
$\mathcal{N}_{68,39}$ & $\mathcal{C}_{68,4}$ & $\left(
0110001101100010100111011010000110\right) $ & 109 \\ \hline
$\mathcal{N}_{68,40}$ & $\mathcal{C}_{68,3}$ & $\left(
1010000001011010101100001001100000\right) $ & 112 \\ \hline
$\mathcal{N}_{68,41}$ & $\mathcal{C}_{68,3}$ & $\left(
0110010111010000001000111001101110\right) $ & 114 \\ \hline
\end{tabular}%
\end{center}
\end{table}

\begin{remark}
We construct $39$ new codes with the rare parameter for
$\boldsymbol{\gamma =5}$ in $W_{68,2}$. Together with these, the
existence of codes with weight enumerator $\boldsymbol{\gamma =5}$
in $W_{68,2}$ is known for $51$ different $\beta $ values.
\end{remark}

\section{Conclusion}
The special structure of the matrices of Baumert-Hall arrays and
more generally orthogonal designs provide a strong link between
discrete structures and self-dual codes. The reduced search field is
instrumental in finding extremal self-dual codes. As has been
demonstrated in the paper, these constructions can be combined with
other methods in the literature such as extensions, search over
rings and neighbors. We have been able to find a substantial number
of new extremal binary self-dual codes using these techniques, thus
filling many gaps in the literature of such codes. Using the
Assmus-Mattson theorem we were also able to come up with new
designs, establishing a key link between self-dual codes and
designs.

The effectiveness of our methods indicate that they can be applied
in different settings as well. We envision two possible directions
for future research. One is to apply the ideas and methods to
different lengths than we have considered. However, it should be
noted that higher lengths would require higher computational power
and so the complexity might become an issue. Another possible idea
is to apply these constructions to other rings than the ones we have
considered.


\begin{thebibliography}{99}
\bibitem{Patrick} A. Alahmadi, C. Guneri, B. Ozkaya, H. Shoaib and P. Sole,  ``On self-dual double negacirculant codes",  \emph{Discrete Applied Math.}, Vol. 222, pp. 205--212, 2017.

\bibitem{anev} D. Anev, M. Harada and N. Yankov, \textquotedblleft New
extremal singly even self-dual codes of lengths 64 and 66", \emph{J.
Algebra Comb. Discrete Appl.}, Vol. 5, No. 3, pp. 143--151, 2017.

\bibitem{baumert} L.D. Baumert and M.Hall, Jr., \textquotedblleft A new
construction for Hadamard matrices", \emph{Bull. Amer. Math. Soc.,}
Vol. 68, pp. 237--238, 1962.

\bibitem{betsumiya} K. Betsumiya, S. Georgiou, T.A. Gulliver, M. Harada and
C. Koukouvinos, \textquotedblleft On self-dual codes over some prime
fields", \emph{Discrete Math}, Vol. 262, pp. 37--58, 2003.

\bibitem{buyuklieva} S. Buyukl\i eva and I. Boukliev, \textquotedblleft
Extremal self-dual codes with an automorphism of order 2",
\emph{IEEE Trans. Inform. Theory}, Vol. 44, pp. 323--328, 1998.

\bibitem{conway} J.H. Conway and N.J.A. Sloane, \textquotedblleft A new upper
bound on the minimal distance of self-dual codes", \emph{\ IEEE
Trans. Inform. Theory}, Vol. 36, No. 6, pp. 1319--1333, 1990.

\bibitem{dougherty1} S.T. Dougherty, T.A. Gulliver and M.
Harada,\textquotedblleft Extremal binary self dual codes",
\emph{IEEE Trans. Inform. Theory}, Vol. 43, pp. 2036--2047, 1997.

\bibitem{frobenius} S.T. Dougherty, J.-L. Kim, H. Kulosman and H. Liu,
\textquotedblleft Self-dual codes over Commutative Frobenius rings",
\emph{Finite Fields Appl.}, Vol. 16, No. 1, pp. 14--26, 2010.

\bibitem{tez} H. Gholamiangonabadi, \textquotedblleft Amicable $T$-matrices
and applications", \emph{MS Thesis}, 2012. Available online:
https://www.uleth.ca/dspace/bitstream/handle/10133/3262/gholamiangonabadi,
hamed.pdf

\bibitem{gildea} J. Gildea, A. Kaya and B. Yildiz, \textquotedblleft An Altered
Four Circulant Construction for Self-Dualcodes from Group Rings and
new extremal binaryself-dual codes I" \emph{in submission}.

\bibitem{goethals} J.M. Goethals and  J.J. Seidel, \textquotedblleft A
skew--Hadamard matrix of order 36" \emph{J. Austral. Math. Soc. A} ,
Vol. 11, pp. 343--344, 1970.

\bibitem{harada} M. Harada and A. Munemasa, \textquotedblleft Some restrictions
on weight enumerators of singly even self-dual codes", \emph{IEEE
Trans. Inform. Theory}, Vol. 52, pp. 1266--1269, 2006.

\bibitem{fsd} S. Karadeniz, S.T. Dougherty and B. Yildiz, ``Constructing Formally Self-Dual Codes over $R_k$", \emph{Discrete Applied Math.}, Vol. 167, pp.188--196, 2014.

\bibitem{kaya} A. Kaya, \textquotedblleft New extremal binary self-dual
codes of length 64 and 66 from $R_{2}$-lifts", \emph{Finite Fields
Appl.}, Vol. 46, pp. 271--279, 2017.

\bibitem{shortkharaghani} A. Kaya, \textquotedblleft New extremal binary
self-dual codes of length 68 via short Kharaghani array over ",
\emph{Math. Commun.}, Vol. 22, pp. 121--131, 2017.

\bibitem{kayayildiz} A. Kaya and B. Yildiz, \textquotedblleft Various
constructions for self-dual codes over rings and new binary
self-dual codes", \emph{Discrete Math}, Vol. 339, No. 2, pp.
460--469, 2016.

\bibitem{pasa} A. Kaya, B. Yildiz, A. Pa\c{s}a \textquotedblleft New
extremal binary self-dual codes from a modified four circulant
construction", \emph{Discrete Math}, Vol. 339, No. 3, pp.
1086--1094, 2016.

\bibitem{kharaghani} H. Kharaghani, \textquotedblleft Arrays for orthogonal
designs", \emph{J. Combin. Des.}, Vol. 8, No. 3, pp.166--173, 2000.

\bibitem{ling} S. Ling and P. Sole, \textquotedblleft Type II codes over $%
\mathbb{F}_{4}+u\mathbb{F}_{4}$", \emph{Europ. J. Combinatorics},
Vol. 22 pp. 983--997, 2001.

\bibitem{Rains} E.M. Rains, \textquotedblleft Shadow Bounds for Self Dual
Codes", \emph{IEEE Trans. Inform. Theory}, Vol. 44, pp.134--139,
1998.

\bibitem{tonchev} V. D. Tonchev, ``Self-dual codes and Hadamard matrices",
\emph{Discrete Applied Math.}, Vol. 33, no. 1–3, pp. 235--240, 1991.

\bibitem{tufekci} N. T\"{u}fek\c{c}i, B. Yildiz, \textquotedblleft On
codes over $R_{k,m}$ and constructions for new binary self-dual
codes", \emph{Mathematica Slovaca}, Vol. 66, No. 6, pp.1511--1526,
2016.

\bibitem{williamson} J. Williamson, \textquotedblleft Hadamard's determinant
theorem and the sum of four squares" \emph{Duke Math. J.}, Vol.11,
pp. 65--81, 1944.

\bibitem{yankov18} N. Yankov, M. Ivanova and M-H. Lee, \textquotedblleft
Self-dual codes with an automorphism of order 7 and s-extremal codes
of length 68" \emph{Finite Fields Appl.}, Vol. 51, pp. 17--30, 2018.
\end{thebibliography}
\end{document}